\documentclass[reqno,11pt]{amsart}
\usepackage{amssymb,amsthm,wasysym}

\usepackage[latin1]{inputenc}
\usepackage[T1]{fontenc}
\usepackage[spanish,english]{babel}

\usepackage{hyperref}
\usepackage[usenames,dvipsnames]{color}

\theoremstyle{plain}

\allowdisplaybreaks

\newtheorem{thm}{Theorem}[section]

\newtheorem{lem}[thm]{Lemma}
\newtheorem{rem}[thm]{Remark}
\newtheorem*{ex*}{Example}

\numberwithin{equation}{section}

\newcommand{\B}{\mathbb{B}}
\newcommand{\N}{\mathbb{N}}
\newcommand{\R}{\mathbb{R}_{+}^{n}}

\def\eps{\varepsilon}
\def\vt{\vartheta}
\def\vr{\varrho}
\def\t{\theta}
\def\1{d \Omega_{\lambda+\mathbf{1}+\eps}(s)}
\def\q{\mathfrak{q}}

\DeclareMathOperator{\supp}{supp}

\begin{document}

\title[Calder\'on-Zygmund operators in the Bessel setting]{Calder\'on-Zygmund operators in the Bessel setting for all possible type indices}

\author[A.J. Castro]{Alejandro J. Castro}
\address{Alejandro J. Castro,     \newline
Departamento de An\'alisis Matem\'atico, Universidad de la Laguna       \newline
Campus de Anchieta, Avda. Astrof\'isico Francisco S\'anchez, s/n,       \newline
38271 La Laguna (Sta. Cruz de Tenerife), Spain
      }
\email{ajcastro@ull.es}

\author[T.Z. Szarek]{Tomasz Z. Szarek}
\address{Tomasz Z. Szarek,     \newline
            Instytut Matematyczny,
      Polska Akademia Nauk, \newline
      \'Sniadeckich 8,
      00--956 Warszawa, Poland
      }
\email{szarektomaszz@gmail.com}

\thanks{The first author was partially supported by MTM2010/17974 and also by an FPU grant from the Government of Spain}

\begin{abstract}
    In this paper we adapt the technique developed in \cite{NoSz} to show that many harmonic analysis operators in the Bessel setting, including
    maximal operators, Littlewood-Paley-Stein type square functions,
    multipliers of Laplace or Laplace-Stieltjes transform type  and Riesz transforms
    are, or can be viewed as, Calderón-Zygmund operators for all possible values of type parameter $\lambda$
in this context.
This extends the results obtained recently in \cite{BCN}, which are valid only for a restricted range of $\lambda$.
\end{abstract}

\maketitle

\section{Introduction and preliminaries}

Let $n \geq 1$ and $\lambda \in (-1\slash 2,\infty)^n$. We consider the Bessel differential operator
$$\Delta_\lambda = -\Delta - \sum_{i=1}^n \frac{2\lambda_i}{x_i} \partial_{x_i},$$
where $\Delta$ stands for the Euclidean Laplacian in $\mathbb{R}_{+}^n = (0,\infty)^n$.
The operator
$\Delta_\lambda$ is symmetric and nonnegative in $C_c^{\infty}(\R) \subset L^2(\R,d\mu_\lambda)$, where $d\mu_\lambda$ is the doubling measure
given by
$$d\mu_\lambda(x)= \prod_{i=1}^n x_i^{2\lambda_i} dx_i, \qquad x \in \R.$$
It is well known that $\Delta_\lambda$ has a self-adjoint extension, here still denoted by $\Delta_\lambda$, whose spectral decomposition is given via the Hankel transform, see \cite{BCN} for details.

The semigroup $\{W_t^\lambda\}_{t>0}$ generated by $-\Delta_\lambda$
has the integral representation
$$W_t^\lambda f(x) = \int_{\R} W_t^\lambda(x,y)f(y) d\mu_\lambda(y),
     \qquad x \in \R, \ t>0,$$
where the Bessel heat kernel is given by
\begin{equation}\label{exp:Bhk}
W_t^\lambda(x,y)=\frac{1}{(2t)^n} \exp\Big( -\frac{1}{4t}\big(|x|^2+|y|^2\big)\Big)
    \prod_{i=1}^n (x_i y_i)^{-\lambda_i+1\slash 2} I_{\lambda_{i}-1/ 2}
        \Big(\frac{x_i y_i}{2t} \Big), \qquad  x,y \in \R, \ t>0,
\end{equation}
with $I_\nu$ denoting the modified Bessel function of the first kind and order $\nu$, cf. \cite[p.\,395]{Wat}.
Note that
$\{W_t^\lambda\}$ is a symmetric diffusion semigroup in the sense of Stein's monograph (\cite[p.\,65]{St1}).\\

In this setting the $n$-dimensional Hankel transform $h_{\lambda}$ plays the same role as the Fourier transform in the Euclidean context.
It is given by
$$ h_{\lambda}f(x) = \int_{\R} \varphi_{x}^{\lambda}(y) f(y)\, d\mu_{\lambda}(y), \qquad x \in \R,$$
with the kernel
\begin{equation*}
\varphi_{x}^{\lambda}(y)
    = \prod_{i=1}^n (x_i y_i)^{-\lambda_i+1\slash 2} J_{\lambda_i-1\slash 2}(x_i y_i), \qquad x,y \in \R,
\end{equation*}
where $J_\nu$ stands for the Bessel function of the first kind and order $\nu>-1$.

We investigate the following multi-dimensional Bessel operators defined initially either in $L^2(d\mu_{\lambda})$ in the cases of (1)-(4), or in $C^{\lambda}$ (the space of smooth $L^2(d\mu_{\lambda})$-functions whose
Hankel transform $h_{\lambda}$ is also smooth and compactly supported) in the case of Riesz transforms (see \cite[Section 4.4]{BCN}).
\begin{itemize}
\item[$(1)$] The maximal operator
$$
W_{*}^{\lambda}f = \big\| W_t^{\lambda}f\big\|_{L^{\infty}(dt)}.
$$
\item[$(2)$] Littlewood-Paley-Stein type mixed square functions
$$
g^{\lambda}_{m,k}(f)(x) = \big\| \partial_x^m \partial_t^k W_t^{\lambda}f (x) \big\|_{L^2(t^{|m|+2k-1}dt)},
$$
where $m \in \N^n$, $k \in \N$, $|m|+k>0$.
\item[$(3)$] Multipliers of Laplace transform type
$$
T^{\lambda}_{\mathcal{M}} f = h_{\lambda}(\mathcal{M} h_{\lambda}f),
$$
where $\mathcal{M}(z) = |z|^2 \int_0^{\infty} e^{-t|z|^2} \psi(t)\, dt$ with $\psi \in L^{\infty}(dt)$.
\item[$(4)$] Multipliers of Laplace-Stieltjes transform type
$$
T^{\lambda}_{\mathcal{M}} f = h_{\lambda} (\mathcal{M} h_{\lambda}f),
$$
where $\mathcal{M}(z) = \int_{(0,\infty)} e^{-t|z|^2} \, d\nu (t)$ with
$\nu$ being a complex Borel measure on $(0,\infty)$.
\item[$(5)$] Riesz transforms of order $m$
$$
R_m^{\lambda}f (x) = \partial_x^m h_{\lambda} (|\cdot|^{-|m|} h_{\lambda} f ) (x),
$$
where $m \in \N^n$ and $|m|>0$.
\end{itemize}

In \cite{BCN} Betancor, Castro and Nowak showed that the above (vector-valued) operators, excluding (4), are Calderón-Zygmund in the sense of the space of homogeneous type
$(\R, d\mu_\lambda, |\cdot|)$, but under the restriction $\lambda \in [0,\infty)^n$. The objective of this paper is to extend that result to the full range of
$\lambda \in (-1/2,\infty)^n$, see Theorem~\ref{thm:main} below.
Typically, the main technical difficulty connected with the Calder\'on-Zygmund approach is to show the relevant kernel estimates.
Here
we follow the ideas of Nowak and Szarek \cite{NoSz}, where they studied several Calderón-Zygmund operators
in the Laguerre setting for all admissible multi-indices of type in that context.

Our results fit into the line of investigations concerning fundamental harmonic analysis operators associated with various discrete and continuous
orthogonal expansions. The starting point of this active area was the influential paper of Muckenhoupt and Stein \cite{MS}.
In that paper harmonic analysis operators such as 
maximal operators, $g$-functions, multipliers and conjugated functions 
were investigated
in the ultraspherical expansions and in the Bessel setting.
Later these operators have been widely examined in the one-dimensional Bessel context by several authors. For instance, following the ideas in \cite{MS},
Andersen and Kerman proved weighted $L^p$ inequalities for the Bessel-Riesz transforms (see \cite{And,AnKer,Ker1,Ker2}), which were lately extended 
in \cite{BBFMT} and \cite{BFMR} to the higher order Riesz transforms by using Calderón-Zygmund theory.
Also Littlewood-Paley-Stein type $g$-functions have drawn considerable attention.
In particular, they were studied by Stempak in \cite{Stem1} 
(more general mixed square functions were treated in \cite{BFS})
and then used to obtain a multiplier theorem similar to the classical Hörmander-Mikhlin theorem.
Laplace transform type multipliers in the Bessel setting were recently analyzed by means of Calderón-Zygmund theory in \cite{BMR} (a general treatment of this kind of multipliers can be found in Stein's monograph \cite{St1}).
It is also worth mentioning the article of Betancor, Harboure, Nowak and Viviani \cite{BHNV}, where they established power weighted $L^p$ mapping properties
for several operators such as heat and Poisson semigroups maximal operators, $g$-functions and Riesz transforms.
However, in all the mentioned papers the one-dimensional situation was investigated. 
It is worth pointing out that only recently the multi-dimensional Bessel context was considered in \cite{BCC1,BCC2,BCDFNR,BCN}.

The paper is organized as follows. Section \ref{sec:mainres} contains the statement of the main result (Theorem \ref{thm:main}) and the reduction of its proof to showing the standard kernel estimates related to the Calder\'on-Zygmund theory. Further, this section is concluded by a remark concerning operators analogous to (1)-(5) and associated with the square root of $\Delta_\lambda$.
Also, some comments pertaining to the considered multipliers are delivered.
Finally, in Section \ref{sec:ker} various preparatory facts are gathered and then the proofs of the relevant kernel estimates are given.

Throughout the paper we use a standard notation with essentially all symbols pertaining to the space of homogeneous type $(\R,d\mu_{\lambda},|\cdot|)$.
However, for any unexplained symbol or notation we refer the reader to \cite{BCN}.
While writing estimates, we will use the notation $X \lesssim Y$ to
indicate that $X \le CY$ with a positive constant $C$ independent of significant quantities.
We will write $X \simeq Y$ when both $X \lesssim Y$ and $Y \lesssim X$ hold.

\section{Main result} \label{sec:mainres}

Let $\mathbb{B}$ be a Banach space and let $K(x,y)$ be a kernel defined on
$\R\times\R\backslash \{(x,y):x=y\}$ and taking values in $\mathbb{B}$.
We say that $K(x,y)$ is a standard kernel in the sense of the space of homogeneous type
$(\R, d\mu_{\lambda},|\cdot|)$ if it satisfies
the growth estimate
\begin{equation} \label{gr}
\|K(x,y)\|_{\mathbb{B}} \lesssim \frac{1}{\mu_{\lambda}(B(x,|x-y|))}
\end{equation}
and the smoothness estimates
\begin{align}
\| K(x,y)-K(x',y)\|_{\mathbb{B}} & \lesssim \frac{|x-x'|}{|x-y|}\, \frac{1}{\mu_{\lambda}(B(x,|x-y|))},
\qquad |x-y|>2|x-x'|, \label{sm1}\\
\| K(x,y)-K(x,y')\|_{\mathbb{B}} & \lesssim \frac{|y-y'|}{|x-y|}\, \frac{1}{\mu_{\lambda}(B(x,|x-y|))},
\qquad |x-y|>2|y-y'| \label{sm2}.
\end{align}
When $K(x,y)$ is scalar-valued, i.e.\ $\mathbb{B}=\mathbb{C}$, the difference bounds \eqref{sm1}
and \eqref{sm2} are implied by the more convenient gradient estimate
\begin{equation} \label{grad}
|\nabla_{\! x,y} K(x,y)| \lesssim \frac{1}{|x-y|\mu_{\lambda}(B(x,|x-y|))}.
\end{equation}
Notice that in these formulas, the ball $B(x,|y-x|)$ can be replaced by $B(y,|x-y|)$, in view of
the doubling property of $\mu_{\lambda}$.

A linear operator $T$ assigning to each $f\in L^2(d\mu_{\lambda})$ a measurable $\B$-valued function $Tf$ on $\R$ is a (vector-valued) Calder\'on-Zygmund operator in the sense of the space $(\R,d\mu_{\lambda},|\cdot|)$ if
\begin{itemize}
    \item[$(i)$] $T$ is bounded from $L^2(d\mu_{\lambda})$ to $L^2_{\B}(d\mu_{\lambda})$,
    \item[$(ii)$] there exists a standard $\B$-valued kernel $K(x,y)$ such that
\begin{align*}
Tf(x)=\int_{\R}K(x,y)f(y)\,d\mu_{\lambda}(y),\qquad \textrm{a.a.}\,\,\, x\notin \supp f,
\end{align*}
for every $f \in L_c^{\infty}(d\mu_{\lambda})$,
where $L_c^{\infty}(d\mu_{\lambda})$ is the subspace of $L^{\infty}(d\mu_{\lambda})$ of bounded measurable functions
with compact supports.
\end{itemize}

To state the main result of the paper, and also for further use, we
denote by $C_0$ a closed separable subspace of $L^{\infty}(dt)$ consisting of all continuous functions on
$\mathbb{R}_{+}$ which have finite limits as $t\to 0^{+}$ and vanish as $t\to \infty$.
\begin{thm}\label{thm:main}
Let $\lambda \in (-1\slash 2,\infty)^n$, $m \in \N^n$, $k\in \N$ be such that $k+|m|>0$, and assume that $\mathcal{M}$ is as in $(3)$ or as in $(4)$ above. Then each of the operators
\[
W^{\lambda}_*, \quad g_{m,k}^{\lambda}, \quad T_{\mathcal{M}}^{\lambda}, \quad R_m^\lambda,
\qquad \textrm{(excluding $m=0$ in the case of $R_m^\lambda$)},
\]
is a (vector-valued) Calder\'on-Zygmund operator
in the sense of the space of homogeneous type $(\R,d\mu_{\lambda},|\cdot|)$
associated with a Banach space $\B$, where $\B$ is $C_0$, $L^2(t^{|m|+2k-1}dt)$, $\mathbb{C}$, $\mathbb{C}$, respectively.
\end{thm}

Using the standard asymptotics for the Bessel function $J_\nu$, $\nu > -1$, (cf. \cite[Chapter III, Section 3$\cdot$1 (8), Chapter VII, Section 7$\cdot$21]{Wat}),
$$
J_{\nu}(z) \simeq z^{\nu}, \quad z \to 0^+, \qquad
J_{\nu}(z) = \mathcal{O}\Big( \frac{1}{\sqrt{z}}\Big), \quad z \to \infty,
$$
we can estimate the one-dimensional kernel of the Hankel transform
\begin{equation*}
    \big|\varphi_{x_i}^{\lambda_i}(y_i)\big|
        \lesssim \left\{
            \begin{array}{rl}
                1, & x_iy_i \leq 1  \\
                (x_iy_i)^{-\lambda_i}, & x_iy_i \geq 1.
            \end{array} \right., \quad i=1, \dots, n.
\end{equation*}
Combining this with the Bessel heat kernel representation \eqref{Bhk} and Lemma~\ref{lem:EST3.3} below, we see that \cite[(13)]{BCN} holds in fact for unrestricted $\lambda \in (-1\slash 2,\infty)^n$. Then the same arguments as those given in \cite{BCN} show that Lemmas 3.5, 3.7 and Remark 3.6 in \cite{BCN} are actually valid for $\lambda \in (-1\slash 2,\infty)^n$.
Consequently, the methods developed in \cite{BCN} to establish
the $L^2(d\mu_\lambda)$-boundedness properties and kernels' associations for $W^{\lambda}_*,$ $g_{m,k}^{\lambda},$ $T_{\mathcal{M}}^{\lambda}$ and $ R_m^\lambda,$
work also in this general case, provided that the standard estimates are true. Note that multipliers of Laplace-Stieltjes transform type were not treated in \cite{BCN}. However, properties
$(i)$ and $(ii)$ above for these operators can be shown essentially in the same way as for the Laplace transform type multipliers.\\

For the sake of clarity and completeness we recall from \cite{BCN} the corresponding Calder\'on-Zygmund kernels and the related Banach spaces.
\begin{itemize}
\item[$(1)$] The kernel associated with the maximal operator is
$$
\mathcal{W}^{\lambda}(x,y) = \big\{W_t^{\lambda}(x,y)\big\}_{t>0}, \qquad \mathbb{B} = C_0 \subset L^{\infty}(dt).
$$
Using \eqref{Bhk} below it can be shown that $\mathcal{W}^{\lambda}(x,y) \in C_0$ for $x\ne y$.
\item[$(2)$] The kernels associated with mixed square functions are
$$
\mathcal{G}^{\lambda}_{m,k}(x,y) = \big\{ \partial_t^k \partial_x^m W_t^{\lambda}(x,y) \big\}_{t>0}, \qquad
    \mathbb{B} = L^2(t^{|m|+2k-1}dt),
$$
where $m \in \N^n$ and $k \in \N$ are such that $|m|+k>0$.
\item[$(3)$] The kernels associated with Laplace transform type multipliers are
$$
K^{\lambda}_{\psi}(x,y) = - \int_0^{\infty} \psi(t) \partial_t W_t^{\lambda}(x,y)\, dt, \qquad
    \mathbb{B}=\mathbb{C},
$$
where $\psi \in L^{\infty}(dt)$.
\item[$(4)$] The kernels associated with Laplace-Stieltjes transform type multipliers are
$$
K^{\lambda}_{\nu}(x,y) =  \int_{(0,\infty)} W_t^{\lambda}(x,y)\, d\nu(t), \qquad
    \mathbb{B}=\mathbb{C},
$$
where $\nu$ are complex Borel measures on $(0,\infty)$.
\item[$(5)$] The kernels associated with Riesz transforms are
$$
R_m^{\lambda}(x,y) = \frac{1}{\Gamma(|m|\slash 2)} \int_0^{\infty} \partial_x^m W_t^{\lambda}(x,y)
    t^{|m|\slash 2 -1}\, dt, \qquad \mathbb{B}=\mathbb{C},
$$
where $m \in \N^n$ is such that $|m| > 0$.
\end{itemize}

Thus to prove Theorem \ref{thm:main}, it suffices to show the following.
\begin{thm}\label{thm:kerest}
Let $\lambda \in (-1\slash 2, \infty)^n$. Then the kernels $(1)$--$(5)$ listed above  satisfy the standard estimates \eqref{gr}, \eqref{sm1} and \eqref{sm2} with the relevant Banach spaces $\B$.
\end{thm}
The proof of this theorem is the most technical part of this paper and is located in Section \ref{sec:ker}.

We end this section with various comments and remarks connected with our main result.

\begin{rem}
   Let $\{ P_t^\lambda \}_{t>0}$ be the Poisson-Bessel semigroup, which is generated by $-\sqrt{\Delta_\lambda}$. By the subordination principle,
    $$ P_t^{\lambda}f(x) = \int_0^{\infty} W^{\lambda}_{t^2\slash (4u)}f(x) \, \frac{e^{-u}du}{\sqrt{\pi u}},
    \qquad x \in \R, \quad t>0.$$
    We consider the maximal operator, Littlewood-Paley-Stein type square functions and multipliers of Laplace or Laplace-Stieltjes transform
    type based on this semigroup, see \cite{BCN} for exact definitions.
Then an analogous result to Theorem \ref{thm:main} is in force also for these operators.
Basically, proving that for every $\lambda \in (-1/2,\infty)^n$ all these Poisson-type operators are (vector-valued) Calderón-Zygmund operators relies on the same arguments as those exposed in the incoming section, see \cite{BCN}, \cite[Section 3]{NoSz} and \cite[Section 4.3]{Sz}, thus we omit the details.
\end{rem}

We now focus on multipliers of Laplace-Stieltjes transform type.
In view of standard arguments it can be shown that these operators can be extended as bounded operators from $L^p(d\mu_{\lambda})$, 
$1 \le p \le \infty$,
into itself. For the reader
convenience we give the details.

Using \cite[Lemma 2.2]{NoSt7} we see that
\begin{equation}\label{Markovian}
\int_{\R} W_t^\lambda (x,y) \, d\mu_\lambda(x)
= \int_{\R} W_t^\lambda (x,y) \, d\mu_\lambda(y)
= 1, \qquad x,y \in \R
\end{equation}
(in fact the Bessel heat semigroup is a Markovian symmetric diffusion semigroup, see \cite[Proposition 6.2]{NoSt7}).
This together with the Fubini theorem gives
\[
\int_{\R} |K^{\lambda}_{\nu}(x,y)|  \, d\mu_\lambda(x)
+
\int_{\R} |K^{\lambda}_{\nu}(x,y)|  \, d\mu_\lambda(y)
\le 2 |\nu|(0,\infty) < \infty,
\]
where $|\nu|$ stands for the total variation of $\nu$.
By standard arguments (cf. \cite[Lemma 2.1]{NoSt7}) we know that the integral operator
$
f \mapsto \int_{\R} K^{\lambda}_{\nu}(x,y) f(y) \, d\mu_\lambda(y)
$
is bounded on $L^p(d\mu_\lambda)$, $1 \le p \le \infty$.
Therefore it is enough to show that, for every $f \in L^2(d\mu_\lambda)$,
\begin{equation}\label{intrepr}
T^{\lambda}_{\mathcal{M}} f (x)
= \int_{\R} K^{\lambda}_{\nu}(x,y) f(y) \, d\mu_\lambda(y),
\qquad \textrm{a.a.}\,\,\, x\in \R
\end{equation}
(notice that Theorem~\ref{thm:main} delivers this only for a.a. $x\notin \supp f$).
Since both sides of \eqref{intrepr} are bounded on $L^2(d\mu_\lambda)$, it suffices to check that
\[
\langle T^{\lambda}_{\mathcal{M}} f , g\rangle_{d\mu_\lambda}
=
\left\langle \int_{\R} K^{\lambda}_{\nu}(x,y) f(y) \, d\mu_\lambda(y) , g\right\rangle_{d\mu_\lambda},
\qquad f,g \in C_c^\infty(\R).
\]
This, however, can easily be verified with the aid of the Fubini theorem, which application is possible by using \eqref{Markovian}.

The advantage of treating these multipliers by means of the Calder\'on-Zygmund theory lies in the fact that we get 
also boundedness on weighted $L^p(d\mu_\lambda)$, $1<p<\infty$, spaces with a large class of weights admitted.
It seems that using arguments similar to those described above it is impossible to obtain such results.

Finally, note that the weak type $(1,1)$ estimate and $L^p(d\mu_\lambda)$-boundedness, $1<p<\infty$, of multipliers of Laplace transform type could be deduced from a general multiplier theorem of Sikora
\cite[Theorem 2.1]{Si}. In that paper a variant of Calder\'on-Zygmund theory (see \cite[Remark on p.\,329]{Si}) is used to obtain the above mentioned results for a general class of multipliers.
However, in comparison with \cite{Si}, we use the standard Calder\'on-Zygmund theory, which allows us to get more results for these operators 
including weighted $L^p(d\mu_\lambda)$ mapping properties
(for more comments about other results we refer the reader to 
\cite[Section 2]{BCN}).

\section{Kernel estimates}\label{sec:ker}
This section delivers proofs of the standard estimates \eqref{gr}-\eqref{sm2} for all the kernels under consideration. We extend the technique applied by Betancor, Castro and Nowak \cite{BCN}, which is valid
for the restricted range of $\lambda \in [0,\infty)^n$.
This method is based on Schl\"afli's integral representation for the modified Bessel function $I_{\nu}$, see \cite[Chapter VI, Section 6$\cdot$15]{Wat} and \cite[(7)]{BCN},
\begin{equation}\label{Iv}
I_{\nu}(z) = z^{\nu} \int_{[-1,1]} \exp({-z s})\, d\Omega_{\nu + 1\slash 2}(s), \qquad
    z>0, \quad \nu \ge -1\slash 2,
\end{equation}
where the measure $\Omega_\eta$ is a product of one-dimensional measures, $\Omega_{\eta}=\bigotimes_{i=1}^{n}\Omega_{\eta_{i}}$, with
$$
d \Omega_{\eta_i}(s_i) = \frac{ (1-s_i^2)^{\eta_i - 1} ds_i } {\sqrt{\pi} 2^{\eta_i - 1\slash 2} \Gamma{ (\eta_{i}) } },
\qquad s_i \in (-1,1),
\qquad \eta_i > 0,
$$
and in the limit case $\Omega_0$ becomes the sum of unit point masses in $1$ and $-1$ divided by $\sqrt{2\pi}$.
Thus under the restriction $\lambda \in [0,\infty)^n$, the Bessel heat kernel can be written as, see \cite[(8)]{BCN},
\begin{equation}\label{restW}
W_t^{\lambda}(x,y)= \frac{1}{(2t)^{n \slash 2 + |\lambda|}}
   \int_{[-1,1]^n}
    \exp\Big(-\frac{1}{4t} q(x,y,s)\Big) \, d\Omega_{\lambda}(s), \qquad x,y\in\R,\quad t>0,
\end{equation}
where $|\lambda|=\lambda_1 + \ldots + \lambda_n$, and the function $q$ is given
by
$$
q(x,y,s) = |x|^2 + |y|^2 + 2 \sum_{i=1}^n x_i y_i s_i, \qquad x,y \in \mathbb{R}_+^n, \quad s \in [-1,1]^n.
$$

Following the ideas from \cite{NoSz}, to express the Bessel heat kernel for the full range of $\lambda \in (-1 \slash 2,\infty)^n$, we use the recurrence relation for $I_{\nu}$, see \cite[Chapter III, Section 3$\cdot$71]{Wat},
\begin{equation*}
I_{\nu}(z) = \frac{2(\nu+1)}{z} I_{\nu+1}(z) + I_{\nu+2}(z).
\end{equation*}
Combining this with \eqref{exp:Bhk} and \eqref{Iv} we arrive at the formula
\begin{equation}\label{Bhk}
W_t^{\lambda}(x,y)= \sum_{\eps \in \{0,1\}^n} C_{\lambda,\eps} \,
    t^{-n\slash 2-|\lambda|-2|\eps|} (xy)^{2\eps} \int_{[-1,1]^n}
    \exp\Big(-\frac{1}{4t} q(x,y,s)\Big) \, \1,
\end{equation}
where $\mathbf{1}=(1,\ldots,1) \in \R$,
$(xy)^{2\eps} = (x_1 y_1)^{2\eps_1} \cdot \ldots \cdot (x_n y_n)^{2\eps_n}$, and
$C_{\lambda,\eps}=(2\lambda+\mathbf{1})^{\mathbf{1}-\eps} \,
2^{-n\slash 2 - |\lambda| - 2|\eps|}$.
This representation turns out to be convenient for our considerations connected with the Calder\'on-Zygmund theory.

To state the lemma below, and also for further use, it is convenient to introduce the following notation.
Given $x,y \in \R$ and $\alpha \in \mathbb{R}^{n}$, we let
\begin{align*}
x^\alpha &= x_1^{\alpha_1} \cdot \ldots \cdot x_n^{\alpha_n}\\
xy &= (x_1 y_1 , \ldots , x_n y_n)\\
x \vee y &= (\max\{x_1, y_1\}, \ldots, \max\{x_n, y_n\})\\
x \le y & \equiv x_i \le y_i, \qquad i=1,\ldots,n.
\end{align*}
We will often neglect the set of integration $[-1,1]^n$ in integrals against $\1$ and write shortly $\q$ instead of $q(x,y,s)$, provided that it does not lead to a confusion.

\begin{lem}[{\cite[Lemma 2.1]{NoSz}}] \label{lem:bridge}
Let $\lambda \in (-1\slash 2,\infty)^n$. Assume that $\xi,\kappa
\in [0,\infty)^n$ are fixed and such that $\lambda + \xi + \kappa
\in [0,\infty)^n$. Then
\begin{align*}
(x+y)^{2\xi} \int \q^{-n\slash 2 - |\lambda|-|\xi|} \, d\Omega_{\lambda+\xi+\kappa}(s)
& \lesssim \frac{1}{\mu_{\lambda}(B(x,|x-y|))},
\end{align*}
uniformly in $x,y \in \mathbb{R}_+^n$, $x\neq y$.
\end{lem}

This technical result, which is a natural generalization of \cite[Proposition 5.9]{NoSt2}, is one of the main points in the whole method of proving kernel estimates. It establishes a relation between
expressions involving certain integrals with respect to
$d \Omega_{\lambda+\mathbf{1}+\eps}(s)$, see \eqref{Bhk}, and the standard estimates
for the space $(\R,d\mu_{\lambda},|\cdot|)$.

It should be noted that for every $\lambda \in (-1\slash 2,\infty)^n$
the $\mu_{\lambda}$-measure of the ball $B(x,R)$ can be described by the same formula as in \cite[Section 3]{BCN}, see \cite[Lemma 2.2]{NoSz},
$$
\mu_{\lambda}(B(x,R)) \simeq R^n \prod_{j=1}^n (x_j + R)^{2\lambda_j},
 \qquad x \in \R, \quad R>0.
$$

To estimate kernels defined via $W_t^\lambda(x,y)$ we will frequently use the following generalization of \cite[Lemma 3.3]{BCN}.
\begin{lem}\label{lem:EST3.3}
Let $W \in \mathbb{R}$, $m,r \in \N^n$, $k \in \N$ and $\eps \in \{0,1\}^n$. Then
\begin{align*}
& \bigg| \partial^k_t \partial_x^m \partial_y^r
\bigg[t^{W}
    (xy)^{2\eps} \exp\Big(-\frac{1}{4t} \q \Big)
\bigg]\bigg| \\
& \quad \lesssim
   \sum_{\beta,\gamma \in \{0,1,2\}^n}
    x^{2\eps-\beta\eps} y^{2\eps-\gamma\eps}
    t^{W-k-(|m|-|\beta\eps|+|r|-|\gamma\eps|)\slash 2}
    \exp\Big(-\frac{1}{8t} \q \Big),
\end{align*}
uniformly in $x,y \in \R$, $t>0$ and $s \in [-1,1]^n$.
\end{lem}

\begin{proof}
    First of all, we observe that
    $$\partial^k_t \partial_x^m \partial_y^r \bigg[t^{W} (xy)^{2\eps} \exp\Big(-\frac{1}{4t} \q \Big) \bigg]
        = \partial^k_t \bigg[t^{W} \prod_{i=1}^n \partial_{x_i}^{m_i} \partial_{y_i}^{r_i} \Big( (x_iy_i)^{2\eps_i} \exp\Big(-\frac{1}{4t} \q_i \Big) \Big) \bigg],$$
    where $\q_i=x_i^2+y_i^2+2x_iy_is_i$, $i=1, \dots, n$. Given $i \in \{1, \dots, n\}$, we distinguish two cases. If $\eps_i=0$, then by \cite[(10)]{BCN}
    we know that
    \begin{align*}
     &   \partial_{x_i}^{m_i} \partial_{y_i}^{r_i} \exp\Big( -\frac{1}{4t}\q_i\Big)\\
        & \quad
            =   \sum_{\substack{0 \le M_i \le m_i\\ 0 \le R_i \le r_i}} P_{m_i,r_i,M_i,R_i}(s_i) \, t^{-(m_i+r_i+M_i+R_i)\slash 2}
             \big( \partial_{x_i}\q_i\big)^{M_i} \big(\partial_{y_i} \q_i\big)^{R_i} \exp\Big( -\frac{1}{4t} \q_i\Big),
    \end{align*}
    where $P_{m_i,r_i,M_i,R_i}$ are polynomials. On the other hand, when $\eps_i=1$, an application of Leibniz' rule and again \cite[(10)]{BCN}
    leads to
    \begin{align*}
        \partial_{x_i}^{m_i} \partial_{y_i}^{r_i} &\Big( (x_iy_i)^{2} \exp\Big(-\frac{1}{4t} \q_i \Big) \Big)\\
            = &  \sum_{\beta_i,\gamma_i \in \{0,1,2\}} C_{m_i,r_i,\beta_i,\gamma_i}
\, \chi_{ \{ \beta_i \le m_i \} }  \chi_{ \{ \gamma_i \le r_i \} }
x_i^{2-\beta_i} y_i^{2-\gamma_i}
\partial_{x_i}^{m_i-\beta_i} \partial_{y_i}^{r_i-\gamma_i} \exp\Big( -\frac{1}{4t}\q_i\Big) \\
            = & \sum_{\beta_i,\gamma_i \in \{0,1,2\}} x_i^{2-\beta_i} y_i^{2-\gamma_i}
                    \sum_{\substack{0 \le M_i \le m_i-\beta_i\\ 0 \le R_i \le r_i-\gamma_i}} P_{m_i,r_i,\beta_i,\gamma_i,M_i,R_i}(s_i) \, t^{-(m_i-\beta_i+r_i-\gamma_i+M_i+R_i)\slash 2} \\
            & \times \big( \partial_{x_i}\q_i\big)^{M_i} \big(\partial_{y_i} \q_i\big)^{R_i} \exp\Big( -\frac{1}{4t} \q_i\Big),
    \end{align*}
    with $C_{m_i,r_i,\beta_i,\gamma_i} \in \mathbb{R}$ and $P_{m_i,r_i,\beta_i,\gamma_i,M_i,R_i}$ being polynomials. Hence,
    \begin{align*}
        \partial_x^m \partial_y^r & \bigg[(xy)^{2\eps} \exp\Big(-\frac{1}{4t} \q \Big) \bigg]
            = \sum_{\beta,\gamma \in \{0,1,2\}^n} x^{2\eps-\beta\eps} y^{2\eps-\gamma\eps} \\
             & \times   \sum_{\substack{0 \le M \le m-\beta\eps\\ 0 \le R \le r-\gamma\eps}} P_{m,r,\beta,\gamma,M,R,\eps}(s) \, t^{-(|m|-|\beta\eps|+|r|-|\gamma\eps|+|M|+|R|)\slash 2}
             \big( \partial_{x}\q\big)^{M} \big(\partial_{y} \q\big)^{R} \exp\Big( -\frac{1}{4t} \q\Big).
    \end{align*}
    Now it remains to take derivatives with respect to $t$. By \cite[(9)]{BCN}, it follows that
    \begin{align*}
        \partial^k_t & \bigg[t^{W} \partial_x^m \partial_y^r \bigg( (xy)^{2\eps} \exp\Big(-\frac{1}{4t} \q \Big) \bigg) \bigg]
            = \sum_{\beta,\gamma \in \{0,1,2\}^n}  x^{2\eps-\beta\eps} y^{2\eps-\gamma\eps}
              \sum_{\substack{0 \le M \le m-\beta\eps\\ 0 \le R \le r-\gamma\eps}} P_{m,r,\beta,\gamma,M,R,\eps}(s) \\
            & \times   \big( \partial_{x}\q\big)^{M} \big(\partial_{y} \q\big)^{R}
               \sum_{0 \leq j \leq k} \alpha_{j,k,W,m,r,\beta,\gamma,M,R,\eps}  \,      t^{W-k-j-(|m|-|\beta\eps|+|r|-|\gamma\eps|+|M|+|R|)\slash 2} \q^j \exp\Big( -\frac{1}{4t} \q\Big),
    \end{align*}
    for some $\alpha_{j,k,W,m,r,\beta,\gamma,M,R,\eps} \in \mathbb{R}$.

Finally, using the estimates
\[
|\partial_{x_i} \q | \lesssim \q^{1/2}, \quad
|\partial_{y_i} \q | \lesssim \q^{1/2}, \qquad
i=1, \dots, n,
\]
and the fact that $\sup_{z \geq 0} z^\alpha e^{-z} < \infty$ for a fixed $\alpha \ge 0$, we get the asserted bound.

\end{proof}

In the next lemma only values of $p \in \{ 1,2,\infty \}$ will be needed for our purposes. However, using the result with $p \in (2,\infty)$ would lead to obtain the
standard estimates for more general Littlewood-Paley-Stein type $g$-functions, for instance $g_{m,k,r}^{\lambda,W}$ investigated in \cite{BCN}.

\begin{lem}\label{lem:bridge2}
Assume that $\lambda \in (-1\slash 2 , \infty)^{n}$, $1 \le p \le \infty$,
$W \in \mathbb{R}$ and $C>0$. Further, let $\eps \in \{ 0,1 \}^{n}$ and $\vt, \vr \in \{ 0,1,2 \}^{n}$
be such that $\vt \le 2 \eps$ and $\vr \le 2 \eps$. Given $u \ge 0$, we consider the function
$\Upsilon_u \colon \R \times \R \times \mathbb{R}_+ \to \mathbb{R}$ defined by
\begin{align*}
 \Upsilon_{u}(x,y,t) =
    t^{ -n \slash 2 - |\lambda| - 2|\eps| + |\vt|\slash 2+|\vr|\slash 2 - W\slash p-u\slash 2}
    \,x^{2\eps-\vt}y^{2\eps-\vr}
\int  \exp\Big(-\frac{C\q}{t} \Big) \, \1,
\end{align*}
where $W \slash p = 0$ for $p=\infty$. Then $\Upsilon_u$ satisfies the integral estimate
\begin{align*}
\big\|\Upsilon_{u}\big(x,y,t\big)\big\|_{L^{p}(t^{W-1}dt)}
\lesssim
\frac{1}{|x-y|^{u}} \;
\frac{1}{\mu_{\lambda}(B(x,|y-x|))}
\end{align*}
uniformly in $x,y \in \R$, $x\neq y$.
\end{lem}

\begin{proof}
We assume that $p<\infty$. The case $p=\infty$ is similar and is left to the reader.

Applying Minkowski's integral inequality and then changing the variable $\q \slash t \mapsto \tau$ and using the inequality $|x-y|^2 \le \q$, we obtain
\begin{align*}
& \|\Upsilon_{u}(x,y,t)\|_{L^{p}(t^{W-1}dt)}\\
& \quad \le
    x^{2\eps-\vt} y^{2\eps-\vr} \int
    \Big\| t^{ -n \slash 2 - |\lambda| - 2|\eps| + |\vt|\slash 2+|\vr|\slash 2 - W\slash p-u\slash 2}
    \exp\Big(-\frac{C\q}{t} \Big) \Big\|_{L^{p}(t^{W-1}dt)} \, \1\\
& \quad =
        x^{2\eps-\vt} y^{2\eps-\vr} \int
    \q^{-n\slash 2 -|\lambda|- 2|\eps| + |\vt|\slash 2+|\vr|\slash 2-u\slash 2}
         \, \1\\
&\qquad \times
    \bigg( \int_0^\infty \tau^{ p ( n\slash 2 +|\lambda| + 2|\eps| - |\vt|\slash 2 - |\vr|\slash 2 + u\slash 2) - 1 } \exp(-Cp\tau) \, d\tau \bigg)^{1\slash p}\\
& \quad \lesssim
    \frac{1}{|x-y|^{u}}
    (x+y)^{2(2\eps-\vt \slash 2 - \vr \slash 2)}
    \int \q^{-n\slash 2-|\lambda|-|2\eps - \vt \slash 2 - \vr \slash 2|}
    \, \1.
\end{align*}
Now the required bound follows by means of Lemma \ref{lem:bridge} specified to
$\xi=2\eps-\vt \slash 2-\vr \slash 2$ and $\kappa=\mathbf{1} - \eps + \vt \slash 2+\vr \slash 2$.
\end{proof}

The two lemmas below will be useful in justifying the smoothness estimates \eqref{sm1} and \eqref{sm2} when the corresponding kernel is not scalar valued
($\B \ne \mathbb{C}$).

\begin{lem}[{\cite[Lemma 4.5]{Sz}}, {\cite[Lemma 4.3]{Sz1}}] \label{lem:theta}
Let $x,y,z\in\R$ and $s \in [-1,1]^n$. Then
$$
\frac{1}{4} q(x,y,s) \le q(z,y,s) \le 4 q(x,y,s),
$$
provided that $|x-y|>2|x-z|$. Similarly, if $|x-y|>2|y-z|$ then
$$
\frac{1}{4} q(x,y,s) \le q(x,z,s) \le 4 q(x,y,s).
$$
\end{lem}

\begin{lem}[{\cite[Lemma 4.5]{Sz1}}] \label{lem:double}
Let $\lambda \in (-1 \slash 2,\infty)^n$. We have
\begin{align*}
\frac{1}{|z-y|\mu_{\lambda}(B(z,|z-y|))}
\simeq
\frac{1}{|x-y|\mu_{\lambda}(B(x,|x-y|))}
\end{align*}
on the set $\{(x,y,z) \in \R \times \R \times \R : |x-y|>2|x-z|\}$.
\end{lem}

To be precise, in \cite[Lemma 4.5]{Sz1} only restricted range of $\lambda \in [0,\infty)^n$ was allowed. However, the same arguments as those in \cite{Sz1} show the result in the general case.

Now we are in a position to prove Theorem \ref{thm:kerest}. In the proof we always tacitly assume that passing with the differentiation in $x_i$, $y_i$ and $t$ under integrals against $\1$, $dt$ or $d\nu(t)$ is legitimate. In fact, such manipulations can easily be justified by using the estimates obtained along the proof of Theorem \ref{thm:kerest} and the dominated convergence theorem.
\begin{proof}[Proof of Theorem \ref{thm:kerest}; the case of $\mathcal{W}^{\lambda}(x,y)$]
Taking into account \eqref{Bhk}, the growth bound \eqref{gr} for $\mathcal{W}^{\lambda}(x,y)$ is a straightforward consequence of Lemma \ref{lem:bridge2} (specified to $u=0$, $p=\infty$, $W=1$, $C=1 \slash 4$, $\vt = \vr =0$).

Next we focus on the smoothness conditions. For symmetry reasons it suffices to verify \eqref{sm1}. An application of the Mean Value Theorem gives
$$
|W_t^{\lambda}(x,y) - W_t^{\lambda}(x',y)|
\le
|x-x'| \Big| \nabla_{\! x} W_t^{\lambda}(x,y) \big|_{x=\t} \Big|,
$$
where $\t = \t(t,x,x',y)$ is a convex combination of $x$ and $x'$. Thus it is enough to show that
$$
\bigg\| \Big| \nabla_{\! x} W_t^{\lambda}(x,y) \big|_{x=\t} \Big| \bigg\|_{L^\infty (dt)}
\lesssim
\frac{1}{|x-y|\mu_{\lambda}(B(x,|x-y|))}, \qquad |x-y|>2|x-x'|.
$$
Differentiating \eqref{Bhk} and then using sequently Lemma \ref{lem:EST3.3} (with $W=-n\slash 2 - |\lambda| - 2|\eps|$, $k=|r|=0$, $m=e_j$, $j=1,\ldots,n$; here $e_j$ is the $j$th coordinate vector in $\mathbb{R}^n$), the inequalities
\begin{equation} \label{est1}
\t \le x \vee x' , \qquad  |x - \t| \le |x-x'| , \qquad
|x - x \vee x'| \le |x - x'|,
\end{equation}
and then Lemma \ref{lem:theta} twice (first with $z = \t$ and then with $z = x \vee x'$) we obtain
\begin{align*}
\Big| \nabla_{\! x} W_t^{\lambda}(x,y) \big|_{x=\t}  \Big|
    & \lesssim
\sum_{\eps \in \{ 0,1 \}^n} \sum_{\beta, \gamma \in \{ 0,1,2 \}^n}
      \t^{2\eps-\beta\eps} y^{2\eps-\gamma\eps}
    t^{-n\slash 2 - |\lambda| -2|\eps| + |\beta\eps|\slash 2 + |\gamma\eps| \slash 2 -1 \slash 2}\\
& \qquad \times \int \exp\Big(-\frac{1}{8t} q(\t,y,s) \Big) \, \1 \\
    &\le
\sum_{\eps \in \{ 0,1 \}^n} \sum_{\beta, \gamma \in \{ 0,1,2 \}^n}
      (x \vee x')^{2\eps-\beta\eps} y^{2\eps-\gamma\eps}
    t^{-n\slash 2 - |\lambda| -2|\eps| + |\beta\eps|\slash 2 + |\gamma\eps| \slash 2 -1 \slash 2}\\
& \qquad \times \int \exp\Big(-\frac{1}{128t} q(x \vee x',y,s) \Big) \, \1 ,
\end{align*}
provided that $|x-y|>2|x-x'|$. This, along with Lemma \ref{lem:bridge2} (taken with $u=1$, $p=\infty$, $W=1$, $C=1 \slash 128$, $\vt = \beta\eps$ and $\vr = \gamma\eps$) and Lemma \ref{lem:double} (with $z= x \vee x'$), gives the desired estimate.
\end{proof}
\begin{proof}[Proof of Theorem \ref{thm:kerest}; the case of $\mathcal{G}^{\lambda}_{m,k}(x,y)$]
Combining Lemma \ref{lem:EST3.3} (applied with $W=-n\slash 2 - |\lambda| - 2|\eps|$, $|r|=0$) with Lemma \ref{lem:bridge2} (taken with $u=0$, $p=2$, $W=|m|+2k$, $C=1 \slash 8$, $\vt = \beta\eps$ and $\vr = \gamma\eps$) leads directly to the growth bound \eqref{gr}.

Proving the smoothness estimates we focus only on \eqref{sm1}. The other bound is justified by analogous arguments. In view of the Mean Value Theorem, it suffices to verify that
$$
\bigg\| \Big| \nabla_{\! x} \partial_{x}^m \partial_{t}^k W_t^{\lambda}(x,y) \big|_{x=\t} \Big| \bigg\|_{L^2 (t^{|m|+2k-1}dt)}
\lesssim
\frac{1}{|x-y|\mu_{\lambda}(B(x,|x-y|))}, \qquad |x-y|>2|x-x'|,
$$
where $\t = \t(t,x,x',y)$ is a convex combination of $x$ and $x'$. Using sequently Lemma \ref{lem:EST3.3} (specified to $W=-n\slash 2 - |\lambda| - 2|\eps|$, $|r|=0$), the inequalities \eqref{est1} and Lemma \ref{lem:theta} twice (with $z = \t$ and then with $z = x \vee x'$) we infer that
\begin{align*}
&\Big| \nabla_{\! x} \partial_{x}^m \partial_{t}^k W_t^{\lambda}(x,y) \big|_{x=\t} \Big| \\
& \quad \lesssim
\sum_{\eps \in \{ 0,1 \}^n} \sum_{\beta, \gamma \in \{ 0,1,2 \}^n}
      (x \vee x')^{2\eps-\beta\eps} y^{2\eps-\gamma\eps}
    t^{-n\slash 2 - |\lambda| -2|\eps| - k - ( |m| - |\beta\eps| - |\gamma\eps| ) \slash 2 -1 \slash 2}\\
& \quad \qquad \times \int \exp\Big(-\frac{1}{128t} q(x \vee x',y,s) \Big) \, \1 .
\end{align*}
Hence, with the aid of Lemma \ref{lem:bridge2} (applied with $u=1$, $p=2$, $W=|m|+2k$, $C=1 \slash 128$, $\vt = \beta\eps$ and $\vr = \gamma\eps$) and Lemma \ref{lem:double} (taken with $z = x \vee x'$), we arrive at the conclusion.
\end{proof}
\begin{proof}[Proof of Theorem \ref{thm:kerest}; the case of $K^{\lambda}_{\psi}(x,y)$]
The growth condition is a simple consequence of Lemma \ref{lem:EST3.3} (specified to $W=-n\slash 2 - |\lambda| - 2|\eps|$, $k=1$ and $|m|=|r|=0$), the fact that $\psi$ is bounded, and Lemma \ref{lem:bridge2} (taken with $u=0$, $p=1$, $W=1$, $C=1 \slash 8$, $\vt = \beta\eps$ and $\vr = \gamma\eps$).

We pass to proving the gradient estimate \eqref{grad}. Since $\psi \in L^{\infty}(dt)$, it is enough to check that
$$
\Big\| \big| \nabla_{\! x,y} \partial_{t} W_t^{\lambda}(x,y) \big| \Big\|_{L^1 (dt)}
\lesssim
\frac{1}{|x-y|\mu_{\lambda}(B(x,|x-y|))}, \qquad x \ne y.
$$
This, however, follows by combining Lemma \ref{lem:EST3.3} (specified to $W=-n\slash 2 - |\lambda| - 2|\eps|$, $k=1$ and $m=e_j$, $|r|=0$ or $|m|=0$, $r=e_j$, $j=1, \dots, n$)
 with Lemma \ref{lem:bridge2} (applied with $u=1$, $p=1$, $W=1$, $C=1 \slash 8$, $\vt = \beta\eps$ and $\vr = \gamma\eps$).
\end{proof}
\begin{proof}[Proof of Theorem \ref{thm:kerest}; the case of $K^{\lambda}_{\nu}(x,y)$]
Since the measure $\nu$ is complex (in particular, its total variation is finite), in order to prove the standard estimates it suffices to verify that
\begin{align*}
\big\| W_t^{\lambda}(x,y) \big\|_{L^{\infty}(dt)}
& \lesssim
\frac{1}{\mu_{\lambda}(B(x,|x-y|))}, \qquad x \ne y, \\
\Big\| \big| \nabla_{\!x,y} W_t^{\lambda}(x,y) \big| \Big\|_{L^{\infty}(dt)}
& \lesssim
\frac{1}{|x-y|\mu_{\lambda}(B(x,|x-y|))}, \qquad x \ne y.
\end{align*}
The first bound here is just the growth condition for $\mathcal{W}^{\lambda}(x,y)$, which is already justified. The second one is implicitly contained in the proof of the smoothness estimates for $\mathcal{W}^{\lambda}(x,y)$.
\end{proof}
\begin{proof}[Proof of Theorem \ref{thm:kerest}; the case of $R_m^{\lambda}(x,y)$]
The growth condition is obtained by using Lemma \ref{lem:EST3.3} (specified to $W=-n\slash 2 - |\lambda| - 2|\eps|$, $k = |r| = 0$) and then Lemma \ref{lem:bridge2} (with $u=0$, $p=1$, $W=|m|\slash 2$, $C=1 \slash 8$, $\vt = \beta\eps$ and $\vr = \gamma\eps$).

To prove the gradient bound \eqref{grad}, it suffices to show that
$$
\Big\| \big| \nabla_{\!x,y} \partial_x^m W_t^{\lambda}(x,y) \big| \Big\|_{L^{1}(t^{|m|\slash 2 -1}dt)}
\lesssim
\frac{1}{|x-y|\mu_{\lambda}(B(x,|x-y|))}, \qquad x \ne y.
$$
This, however, follows by using Lemma \ref{lem:EST3.3} (taken with $W=-n\slash 2 - |\lambda| - 2|\eps|$, $k = |r| = 0$) and then Lemma \ref{lem:bridge2} (applied with $u=1$, $p=1$, $W=|m|\slash 2$, $C=1 \slash 8$, $\vt = \beta\eps$ and $\vr = \gamma\eps$).
\end{proof}

\end{document}